\documentclass{amsart}
\usepackage{amsmath}
\usepackage{graphicx}
\usepackage{latexsym}
\usepackage{color}
\usepackage{amscd}
\usepackage[all]{xy}
\usepackage{enumerate}
\usepackage{hyperref}
\usepackage{subfigure}
\usepackage{soul}
\usepackage{comment}
\usepackage{amsthm}
\usepackage{amsfonts}  
\usepackage{amssymb}   
\usepackage{url}

\newtheorem{thm}{Theorem}

\newtheorem{theorem}[thm]{Theorem}
\newtheorem{lemma}[thm]{Lemma}

\newtheorem{proposition}[thm]{Proposition}

\newtheorem*{mthma}{Theorem~A}
\newtheorem*{mthmb}{Theorem~B}

\theoremstyle{definition}

\newtheorem*{definition*}{Definition}

\newtheorem{remark}[thm]{Remark}

\newcommand{\N}{\mathbb{N}}
\newcommand{\Z}{\mathbb{Z}}

\newcommand{\K}{{\rm K3}}

\newcommand{\co}{\mskip0.5mu\colon\thinspace}

\def \x {\times}

\begin{document}

\title[Fillings of genus--$1$ open books and $4$--braids]
{Fillings of genus--$1$ open books and $4$--braids}

\author[R. \.{I}. Baykur]{R. \.{I}nan\c{c} Baykur}
\address{Department of Mathematics and Statistics, University of Massachusetts, Amherst, MA 01003-9305, USA}
\email{baykur@math.umass.edu}

\author[J. Van Horn-Morris]{Jeremy Van Horn-Morris}
\address{Department of Mathematical Sciences, The University of Arkansas, \newline 
\indent Fayetteville, AR 72701,  USA }
\email{jvhm@uark.edu }

\begin{abstract}
We show that there are contact $3$--manifolds of support genus one which admit infinitely many Stein fillings, but do not admit arbitrarily large ones. These Stein fillings arise from genus--$1$ allowable Lefschetz fibrations with distinct homology groups, all filling a fixed minimal genus open book supporting the boundary contact $3$--manifold. In contrast, we observe that there are only finitely many possibilities for the homology groups of Stein fillings of a given contact \mbox{$3$--manifold} with support genus zero. We also show that there are $4$--strand braids which admit infinitely many distinct Hurwitz classes of quasipositive factorizations, yielding in particular an infinite family of knotted complex analytic annuli in the \mbox{$4$--ball} bounding the same transverse link up to transverse isotopy. These realize the smallest possible examples in terms of the number of boundary components a genus--$1$ mapping class and the number of strands a braid can have with infinitely many positive/quasipositive factorizations.
\end{abstract}

\maketitle

\setcounter{secnumdepth}{2}
\setcounter{section}{0}

\section{Introduction} 

Open books have gained a prominent role in contact geometry and low dimensional topology ever since Giroux established a striking correspondence between contact structures and open books on $3$--manifolds \cite{Gi2}. It is desirable to deduce contact geometric information from coarse topological invariants of supporting open books. The minimal genus for a supporting open book, called the \emph{support genus} of a contact $3$--manifold $(Y, \xi)$ \cite{EtnyreOzbagci}, is such an invariant. Thanks to works of Etnyre, Wendl, and several others, quite a lot is now known about contact \mbox{$3$--manifolds} with support genus zero, as well as their symplectic and Stein fillings (e.g. \cite{Etnyre, Wendl, PV, Kaloti}). In particular, it is known that not all contact structures can be supported by planar open books \cite{Etnyre}. However, 15 years after Giroux, it is still an open question whether there are contact structures which cannot be supported by genus--$1$ open books. 

Naturally, one would like to see if there are any distinguishing properties for contact $3$--manifolds that can be supported by genus--$1$ open books, perhaps similar to those known to hold for planar ones. There are a few intimately related aspects we will consider in this note: to date, there are no known examples of planar open books (resp. planar contact structures) that can be filled by infinitely many allowable Lefschetz fibrations (resp. Stein surfaces), whereas there are many examples of higher $g \geq 2$ open books with such fillings (e.g. \cite{OS, DKP, BMV}). On the other hand no genus--$0$ or genus--$1$ open book can be filled by allowable Lefschetz fibrations with \emph{arbitrarily large} Euler characteristic (see, for example, \cite{BMV}), whereas many examples have been found in recent years again for $g \geq 2$ open books \cite{BV1, BV2, DKP, BMV}. Notably, the Euler characteristic of fillings of a contact $3$--manifold with support genus zero is also bounded \cite{P, Kaloti}, but this property is not known to extend to support genus one. (If it did, this would provide an obstruction for contact $3$--manifolds with arbitrarily large Sten fillings to admit genus--$1$ open books.)
In this note, we will explore each of these types of fillings. We prove the following:

\begin{mthma}
There are (infinitely many) contact $3$--manifolds with support genus one, each one of which admits infinitely many homotopy inequivalent Stein fillings, but do not admit arbitrarily large ones. These are all supported by genus--$1$ open books, each bounding infinitely many distinct genus--$1$ allowable Lefschetz fibrations on their Stein fillings with infinitely many distinct homology groups. 
\end{mthma}

\noindent Many earlier examples of contact $3$--manifolds with infinitely many Stein fillings (e.g. \cite{OS, AEMS}) have used higher genera open books, which also admit arbitrarily large Stein fillings by \cite{BMV, DKP}. Examples of contact $3$--manifolds with support genus one and admitting infinitely many Stein fillings were given by Yasui in \cite{Yasui}, where logarithmic transforms were used to produce exotic fillings (see Remark~\ref{YasuiExotic}). Our theorem provides the first examples which illustrate that the two classes of contact $3$--manifolds are different. 

In the course of proving the above theorem, we will observe that fillings of any given planar contact $3$--manifold not only have restricted Euler characteristics, but also they can have only finitely many \emph{homology groups} (Proposition~\ref{planarfillings}). In contrast, we will produce infinitely many factorizations of genus--$1$ mapping classes into only \emph{four} positive Dehn twists, which induce homologically distinct allowable Lefschetz fibrations (Lemma~\ref{keyexample}). Two ingredients in this simple construction are suitable \emph{partial conjugations} which effectively change the homology (e.g. \cite{OS, BK, B}), and the explicit monodromy of a genus--$1$ open book with three binding components supporting the standard Stein fillable contact structure on $T^3$ (Lemma~\ref{keyrelation} and Proposition~\ref{3torusOB}; also \cite{VHMthesis}). The small topology of these fillings will then allow us to show that many of these contact $3$--manifolds admit \emph{symplectic Calabi-Yau caps} introduced in \cite{LiMakYasui}, so as to conclude that they can \emph{not} admit arbitrarily large Stein fillings (Theorem~\ref{mainthm1}). \linebreak We moreover observe that there are contact $3$--manifolds supported by planar \emph{spinal open books} \cite{LVW, BV1} which also admit infinitely many Stein fillings, but do not admit arbitrarily large ones (Remark~\ref{spinal}).

\smallskip
The second part of our article concerns fillings of \emph{quasipositive links}, which are closures of quasipositive braids. By the pioneering works of Rudolph \cite{Rudolph} and Boileau and Orevkov \cite{BO}, these links are characterized as oriented boundaries of smooth pieces of complex analytic curves in the unit $4$--ball ---which can be realized as \emph{braided surfaces}. An intimate connection between fillings of contact $3$--manifolds and quasipositive braids is provided by Loi and Piergallini \cite{LP}, who showed that every Stein filling of a contact $3$--manifold is a branched cover of the Stein $4$--ball along a braided surface. 

We will provide infinitely many complex curves in the $4$--ball filling the same transverse links in the boundary $3$--sphere:

\begin{mthmb}
There are (infinitely many) elements in the $4$--strand braid group, each of which admits infinitely many quasipositive braid factorizations up to Hurwitz equivalence. A particular family gives infinitely many complex analytic annuli in the $4$--ball with different fundamental group complements, all filling the same \mbox{$2$-component} transverse link.
\end{mthmb}

\noindent We will derive these examples from a refined construction of positive factorizations of elements in the mapping class group of a torus with two boundary components. These factorizations will commute with the hyperelliptic involution exchanging the two boundary circles, so they descend to quasipositive factorizations of certain \mbox{$4$--braids} (Theorem~\ref{mainthm2}). In turn, the quasipositive factorizations prescribe braided surfaces, which by the work of Rudolph, can be made complex analytic. The frugality of our factorizations makes it possible to strike remarkably small topology among these examples, such as complex \emph{annuli} filling in the same $2$--component transverse link. 

The \emph{infinite} families we obtain in this paper drastically improve the existing literature on fillings of genus--$1$ open books and $4$--braids. Our examples demonstrate that for each $k \geq 2$, there are elements in the mapping class group of a torus with $k$ boundary components which admit infinitely many distinct factorizations into positive Dehn twists, distinguished by their homologies. For $k=1$, examples admitting a few distinct factorizations were obtained by Auroux in \cite{Auroux} using calculations in $SL(2, \Z)$, whereas examples of quasipositive factorizations of some $3$-- and $4$--braids yielding \emph{two} distinct braided surfaces were obtained in \cite{AurouxEtal, Geng, Auroux} (see Remarks~\ref{AurouxFact} and ~\ref{GengFact}). Moreover, recent work of Orevkov \cite{O1} shows that any $3$--braid admits at most finitely many quasipositive factorizations up to Hurwitz equivalence, and the
same goes for positive factorizations of a mapping class on a torus with one boundary \cite{CadavidEtal}. Hence our examples of infinite families of fillings are optimal in terms of the number of boundary components for genus--$1$ mapping classes and the number of strands for braids.

\vspace{0.1in}
\noindent \textit{Acknowledgements.} The first author was partially supported by the NSF Grant DMS-$1510395$, and the second author by the Simons Foundation Grant $279342$. We would like to thank Stepan Orevkov for pointing out his result in \cite{O1}.

\vspace{0.1in}
\section{Background}

All manifolds in this paper are assumed to be smooth, compact and oriented. 

We denote a genus $g$ surface with $n$ boundary components by $\Sigma_g^n$, and its \emph{mapping class group} by $\Gamma_g^n$. This is the group which consists of orientation-preserving homeomorphisms of $\Sigma_g^n$ that restrict to identity along $\partial \Sigma_g^n$, modulo isotopies of the same type.  We denote by $t_c \in \Gamma_g^n$, the positive (right-handed) \emph{Dehn twist} along the simple closed curve $c \subset \Sigma_g^n$. A factorization in $\Gamma_g^n$
\[
\phi= t_{c_l} \cdots t_{c_1} \]
 is called a \emph{positive factorization} $P$ of $\phi$. Of particular interest to us here are the factorizations where the Dehn twist curves $c_j$ are all homologically essential on $\Sigma_g^n$. Here we use the functorial notation for the product of mapping classes; e.g. a given factorization $P=t_{c_l} \cdots t_{c_1}$ acts on a curve $a$ on $\Sigma_g^n$ first by $t_{c_1}$, then $t_{c_2}$, and so on.  The notation $\phi^{\eta}$ will stand for the conjugate element $\eta^{-1} \phi \, \eta$. Lastly, we note two elementary facts regarding Dehn twists we will repeatedly use: for $\phi= t_c$, $t_c^{\eta}=t_{\eta (c)}$ and if $a$ and $b$ are disjoint curves, $t_a \, t_b= t_b \, t_a$. 

Let $\mathcal{B}_n$ denote the $n$-strand \emph{braid group}, which consists of orientation-preserving homeomorphism of the unit disk $D^2$ fixing setwise $n$ distinguished marked points in the interior, modulo isotopies of the same type.  We denote by $\tau_{\alpha_j}$  the positive (right-handed) half-twist along an arc $\alpha_j$ between two marked points on $D^2$, avoiding the others. A factorization in $\mathcal{B}_n$ 
\[ b= \tau_{\alpha_l} \cdots \tau_{\alpha_1} \]
is then called a \emph{quasipositive factorization} of the braid $b$. A link in the $3$--sphere is said to be \emph{quasipositive} if it can be realized as the closure of a quasipositive braid.

\subsection{Contact $3$--manifolds and supporting open books} \

A \emph{contact structure} on a $3$--manifold $Y$ is a plane field $\xi$, which can be globally written as the kernel of a  \emph{contact form} $\alpha \in \Omega^1(Y)$ with $\alpha \wedge d\alpha\neq 0$. An \emph{open book} on $Y$ consists of $L$, an $n$-component oriented link in $Y$ which is \emph{fibered} with bundle map $h\co Y \setminus L \to S^1$ such that $\partial F_t = L$ for all $F_t=h^{-1}(t)$. Here $L$ is called the \emph{binding}, the surface $F_t \cong \Sigma_g^n$, for any $t$, is called the \emph{page} of the open book, and in this case $h$ is called a \emph{genus--$g$ open book.} An open book is determined up to isomorphism by an element $\phi \in \Gamma_g^n$, called the \emph{monodromy}, which prescribes the return map of a flow transverse to the pages and meridional near the binding. 

A \emph{contact $3$--manifold} $(Y, \xi)$ is said to be \emph{supported by} or \emph{compatible with} an open book $h$ if $\xi$ is isotopic to a contact structure given by a $1$--form $\alpha$ satisfying $\alpha>0$ on the positively oriented tangent planes to $L$ and $d\alpha$ is a volume form on every page. By the works of Thurston and Winkelnkemper \cite{TW} and Giroux \cite{Gi2}, every open book on $Y$ supports some contact structure, and conversely, there are (infinitely many) open books supporting a given contact structure on $Y$. The \emph{support genus} of a contact $3$--manifold $(Y, \xi)$ is then defined to be the smallest possible genus for an open book supporting it. 

\subsection{Stein fillings, allowable Lefschetz fibrations and braided surfaces} \

A special class of contact $3$--manifolds arise as the boundaries of affine complex $4$--manifolds. A \emph{Stein filling} of a contact $3$--manifold $(Y, \xi)$ is a $4$--manifold $(X,J)$ equipped with an affine complex structure $J$ in its interior, where the maximal complex distribution along $\partial X = Y$ is $\xi$. 

A \emph{Lefschetz fibration} on a $4$--manifold $X$ is a map $f\co X \to D^2$, where each critical point of $f$ lies in the interior of $X$ and conforms to the local complex model  $f(z_1,z_2)=z_1 z_2$ under orientation preserving charts. These singularities are obtained by attaching $2$--handles to a regular fiber along simple-closed curves $c_j$, called the \emph{vanishing cycles}. A Lefschetz fibration is said to be \emph{allowable} if the fiber has non-empty boundary and all its vanishing cycles are homologically nontrivial on $F \cong \Sigma_g^n$, $n>0$. Given an allowable Lefschetz fibration $f\co X \to D^2$, if we let $p$ be a regular value in the interior of the base $D^2$, then composing $f$ with the radial projection $D^2 \setminus
\{p\} \to \partial D^2$ we obtain an open book $h\co \partial X \setminus \partial f^{-1}(p) \to S^1$. Importantly, any positive factorization $P$ of an open book monodromy $\phi = t_{c_l} \cdots t_{c_1}$ in $\Gamma_g^n$ prescribes an allowable Lefschetz fibration $f$ on a Stein filling $(X, J)$ of $(Y, \xi)$ supported by this open book,

A \emph{braided surface} in the $4$--ball is an embedded compact surface $S$, along which the canonical projection $\pi: D^4 \cong D^2 \x D^2 \to D^2$ restricts to a simple positive branched covering, where $S$ and $\pi|_S$ conform to the local complex models $w =z^2$ and $\pi(w, z)=z$ around each branch point. Any quasipositive factorization of a braid $b=\tau_{\alpha_l} \cdots \tau_{\alpha_1}$ in $\mathcal{B}_n$ prescribes a braided surface $S$.  

All these geometric objects come together in a beautiful theorem of Loi and Piergallini.
Building on Eliashberg's topological characterization of Stein fillings and the work of Rudolph on braided surfaces, they showed that every Stein filling come from an allowable Lefschetz fibration on $X$ \cite{LP} (also see \cite{AO}, and \cite{BV1} for a further generalization to Lefschetz fibrations over arbitrary compact surfaces with non-empty boundaries), which arise as a branched covering of the Stein $4$--ball along a braided surface $S$. On the other hand, the works of Rudolph \cite{Rudolph} and Boileau and Orevkov \cite{BO} established that quasipositive links are precisely those which are oriented boundaries of smooth pieces of complex analytic curves in the unit $4$--ball, realized as braided surfaces. 

The algebraic topology of a Stein filling $X$ equipped with an allowble Lefschetz fibration is easy to read off from the corresponding positive factorization $P$. In particular we have
\[ \pi_1(X) \cong \, \pi_1(\Sigma_g^n) \, / \, N(c_1, \ldots, c_l) \, , \]
where $N(c_1, \ldots, c_l)$ is the subgroup of $\pi_1(\Sigma_g^n)$ generated normally by $\{c_i\}$. For $\{a_j\}$ chosen generators of $\pi_1(\Sigma_g^n) \cong \Z^{2g+n-1}$, we therefore get
\[ \pi_1(X) \cong \, \langle \, a_1, \ldots, a_{2g+n-1} \, | \, R_1, \ldots, R_l \, \rangle \, , \]
where each $R_i$ is a relation obtained by expressing $c_i$ in  $\{a_j\}$.

\vspace{0.1in}
\section{Preliminary results}
 
We begin with a relation in the genus--$1$ mapping class group, which will play a key role in our constructions to follow.

\begin{lemma} \label{keyrelation}
The following relation holds in the mapping class group $\Gamma_1^3$:
\[
t_{b_1} t_{b_2} t_{b_3} =  t_{a_1}^{-3} t_{a_2}^{-3} t_{a_3}^{-3} t_{\delta_1} t_{\delta_2} t_{\delta_3} \, ,\] where the curves $a_i, b_i, \delta_i$, $i=1, 2, 3$, are as shown in Figure~1.
\end{lemma}

\begin{figure}[ht!] \label{curves}
  \begin{center}
   \includegraphics[width=13cm]{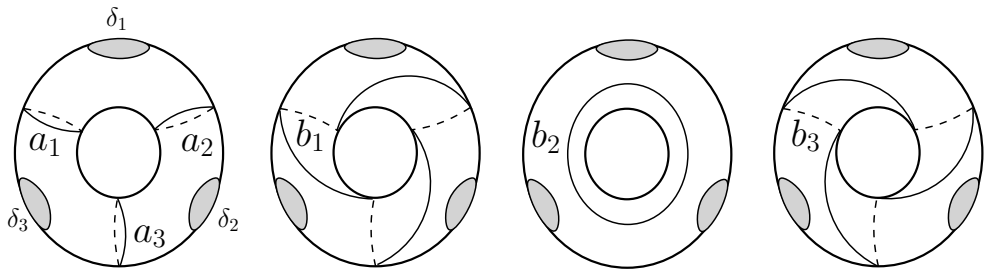}
    \end{center}
   \smallskip
   \caption{Dehn twist curves $a_i$, $b_i$, $\delta_i$ on $\Sigma_1^3$.}
   \end{figure}

\begin{proof}
This follows from the  following relation in $\Gamma_1^3$, known as the \mbox{\emph{star relation} \cite{Gervais}:}
\[ (t_{a_1} t_{a_2} t_{a_3} t_{b_2})^{3} = t_{\delta_1} t_{\delta_2} t_{\delta_3} \, , \]
 which can be derived by applying the lantern relation and braid relations to the well-known $3$-chain relation \cite{KorkmazOzbagci}. For $\Delta= t_{\delta_1} t_{\delta_2} t_{\delta_3}$ the boundary multitwist, and $\eta= t_{a_1} t_{a_2} t_{a_3}$, we can rewrite it as
\begin{align*}
\eta \, t_{b_2} \eta \, t_{b_2} \eta \, t_{b_2} &= \Delta  \\
 (\eta^{-1} \, t_{b_2} \eta) \, t_{b_2} \, (\eta \,  t_{b_2} \eta^{-1})&= \eta^{-3} \Delta \\
t_{\eta^{-1}(b_2)} \, t_{b_2} \, t_{\eta (b_2)} &= \eta^{-3} \Delta \\
t_{b_1} t_{b_2} t_{b_3} &= \eta^{-3} \Delta
\end{align*}
noting that $\Delta$ commutes with $\eta$. Since $a_1, a_2, a_3$ are all disjoint, $\eta^{-3}=  t_{a_1}^{-3} t_{a_2}^{-3} t_{a_3}^{-3}$, which gives us the desired relation. 
\end{proof} 

As we show below, the mapping class above prescribes an important open book. A detailed proof of the next proposition can be found in \cite{VHMthesis}. Here we will give an alternate argument. 

\begin{proposition} \label{3torusOB}
The open book with monodromy \ $\psi = t_{b_1} t_{b_2} t_{b_3}$ supports $(T^3, \xi_{\rm{can}})$, where $\xi_{\rm{can}}$ is the canonical contact structure on the $3$--torus $T^3$, the unit cotangent bundle of the $2$--torus $T^2$. 
\end{proposition}

\begin{proof} 
The positive factorization \,$t_{b_1} t_{b_2} t_{b_3}$ prescribes an allowable Lefschetz fibration whose boundary is the open book with monodromy $\psi$, which implies that the boundary contact $3$--manifold is Stein fillable. 

\begin{figure}[ht!] \label{calculus}
  \begin{center}
   \includegraphics[width=12cm]{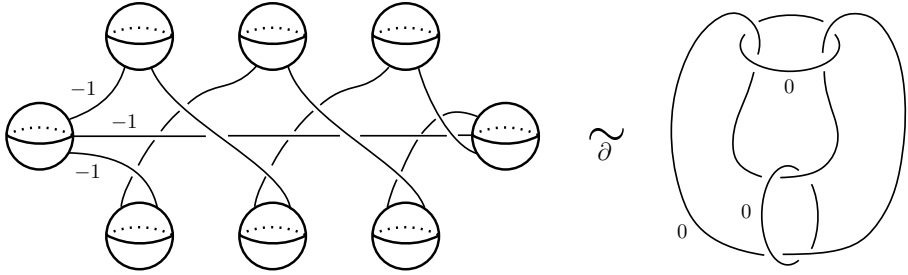}
    \end{center}
   \smallskip
   \caption{Kirby diagram for the Lefschetz fibration filling $Y$, whose boundary is diffeomorphic to $T^3$.}
   \end{figure}

The handlebody decomposition for this fibration yields the Kirby diagram in Figure~2. It is now a Kirby calculus exercise, which we will  leave the details of to the reader. Sliding two of the $(-1)$--framed $2$--handles over the third one we obtain two pairs of canceling $1$-- and $2$--handle pairs we can remove. The resulting diagram is easily seen to be the standard diagram for $D^2 \x T^2$ with two $1$--handles and one $2$--handle, whose boundary is $T^3$. Or instead, one can obtain a surgery diagram for the boundary \mbox{$3$--manifold} after trading the $1$--handles with $0$--framed unknots in the first picture, and perform the similar link calculus to arrive at the surgery diagram for $T^3$ given on the right hand side of Figure~2.

\enlargethispage{0.1in}
Now by Eliashberg \cite{Eliashberg}, the only contact structure on $T^3$ admitting a Stein filling is the canonical structure $\xi_{\rm{can}}$.
\end{proof}

\smallskip
\begin{remark}\label{elliptic}
The star relation prescribes a genus--$1$ elliptic fibration with three \mbox{$(-1)$--sections} \cite{KorkmazOzbagci}. The lemma  therefore implies that we can construct the Stein filling of $T^3$ as the complement of the union of an $I_9$ fiber (which one gets by clustering the nodal singularities induced by all the vanishing cycles $a_i$) and three $(-1)$--sections.
\end{remark}

Before we move on to our construction, we record the following result, which might be of particular interest. 

\begin{proposition} \label{planarfillings}
Let $P$ be a positive factorization of a mapping class \mbox{$\psi \in \Gamma_0^n$} into Dehn twists along homologically essential curves. Let $(X_P, J_P)$ denote the corresponding genus-$0$ allowable Lefschetz fibration. Then {$\{ H_i(X_P) \, | \,  \psi=P \}$} is a {finite set.} It follows that finitely many groups arise as homology groups of all possible Stein fillings (indeed, all minimal symplectic fillings) of a fixed contact \mbox{$3$--manifold} supported by a planar open book. 
\end{proposition}

\begin{proof}
For $\psi = P$ in $\Gamma_0^n$, $X=X_P$ admits a handle decomposition with one \mbox{$0$--handle,} $n$ $1$--handles, and $\ell$ $2$--handles, where $\ell$ is equal to the number of Dehn twists in $P$. We thus have $\widetilde{H}_i(X ; \Z)=0$ for  $i \neq 1, 2$. 

From \cite{P}, proof of Theorem~2.2, (cf. \cite{Kaloti}, proof of Theorem~1.1), we see that for a planar mapping class element $\psi$, the number of Dehn twists in any positive factorization of $\psi$ (that is, $\ell$) is uniformly bounded. This shows that $H_2(X; \Z)$ has a finite number of possibilities. Further, the collection of curves $\{c_i\}$ in $P$ gives a presentation of $H_1(X; \Z)$ with $n$ generators and $\ell$ relations, each of the type 
\[m_{j,1} \, a_1  + \ldots + m_{j,n} \, a_n = 0 \]
Each relator is given by writing the homology class of the curve $c_j$ in the presentation $P$ in the basis $\{a_i\}$, where $a_i$ is the corresponding homology generator in $H_1(X; \Z)$ for each boundary component $\delta_i$, and more importantly where \linebreak $0 \leq | m_{j,i} | \leq 1$. Therefore there are only finitely many possible relations in a presentation of $H_1(X; \Z)$ for a Stein filling $X$ of $\xi$.

The remainder of the proposition follows from work of Wendl \cite {Wendl} and \linebreak Niederkr\"uger-Wendl \cite{NW} equating  minimal strong symplectic fillings with positive factorizations for planar contact manifolds. Namely, any such filling of a contact \mbox{$3$--manifold} with planar supporting open book admits an allowable Lefschetz fibration with planar fibers.
\end{proof}

\vspace{0.1in}
\section{Constructions of infinitely many fillings}

The contact $3$--manifold $(Y, \xi)$ we are interested is the one derived from $(T^3, \xi_{\rm{can}})$ by a Weinstein handle attachment along $b_2$, which by \cite{Gay} is supported by the open book with monodromy $\phi= \psi \, t_{b_2}= t_{b_1} t_{b_2} t_{b_3} t_{b_2}$. Moreover, by \cite{Wendl}, the Stein filling prescribed by the new positive factorization $ t_{b_1} t_{b_2} t_{b_3} t_{b_2}$ is obtained by a Weinstein handle attachment to the unique Stein filling $(D^2 \x T^2, J_{\rm{can}})$ of $(T^3, \xi_{\rm{can}})$. 

Next is the construction of Stein fillings of $(Y, \xi)$ with distinct homologies:

\begin{lemma} \label{keyexample}
Let $(Y, \xi)$ be the contact $3$--manifold supported by the genus--$1$ open book with three boundary components, whose monodromy is $\phi= t_{b_1} t_{b_2} t_{b_3} t_{b_2}$. 
For any integer $n$, the mapping class $\phi$ admits a positive  factorization 
\[P_n=   (\, t_{b_1} t_{b_2} t_{b_3}  \, t_{t_{a_1}^n(b_2)} \,)^{t_{a_1}^{-n}} =
 t_{t_{a_1}^{-n}(b_1)} t_{t_{a_1}^{-n}(b_2)} t_{t_{a_1}^{-n}(b_3)}  \, t_{b_2}  \, .\]
Let $(X_n, J_n)$ be the Stein filling which is the total space of the allowable Lefschetz fibration prescribed \mbox{by $P_n$.}  The family $ \{ \, (X_n, J_n) \, | \, n \in \N \}$ consists of distinct Stein fillings of the contact $3$--manifold $(Y, \xi)$, distinguished by their first homology \mbox{$H_1(X_n; \Z)=  \Z  \oplus \, \Z / n\Z$.}
\end{lemma}

\begin{proof}
In the mapping class group $\Gamma_1^3$, we have
\begin{align*}
t_{b_1} t_{b_2} t_{b_3} \, t_{b_2}^{t_{a_1}^n}   = t_{a_1}^{-n} \, ( \,(t_{b_1} t_{b_2} t_{b_3})^{t_{a_1}^{-n}} \,  t_{b_2}  \, ) \,t_{a_1}^n =  t_{a_1}^{-n} \, (t_{b_1} t_{b_2} t_{b_3}  t_{b_2}  \, ) \,t_{a_1}^n   \, .
\end{align*}
Note that the mapping class $t_{b_1} t_{b_2} t_{b_3} = t_{a_1}^{-3} t_{a_2}^{-3} t_{a_3}^{-3} t_{\delta_1} t_{\delta_2} t_{\delta_3}$ commutes with $t_{a_1}^n$, and thus it stays fixed under the conjugation by $t_{a_1}^{-n}$. We conclude that each $P_n$ is a positive factorization of  $\phi = t_{b_1} t_{b_2} t_{b_3}  t_{b_2}$. 

We can calculate $H_1(X_n; \Z)$ using the positive factorization $P_n$, or the conjugate factorization  $P_n ^{t_{a_1}^{n}}$ (which prescribes an isomorphic Lefschetz fibration), that is
\[ H_1(X_n; \Z) \cong H_1(\Sigma_1^3 ; \Z) \, / N \]
where $N$ is normally generated by the vanishing cycles $b_1, b_2, b_3$ and $t_{a_1}^n(b_2)$. 

The right hand side is generated by the homology classes of the curves $a_1, a_2, a_3, b_2$ on $\Sigma_1^3$ (which we will denote by the same letters), with the relations 
\[  b_2+a_1+a_2+a_3=0 \, , \ b_2=0 \, , \ b_2 -a_1-a_2-a_3=0 \, ,   \ b_2 -n\,a_1=0  \]
induced by the vanishing cycles $b_1, b_2, b_3$ and {$t_{a_1}^{-n}(b_2)$}, where the last one is easily calculated by the Picard-Lefschetz formula. For $n \in \N$, from $b_2=0$, $a_3=-a_1-a_2$ and $n \, a_1=0$, we get $H_1(X_n; \Z) = \Z  \oplus \, \Z /  n\Z$, generated by $a_2$ and $a_1$.

Lastly, our claim on the support genus of $(Y, \xi)$ follows from Proposition~\ref{planarfillings}.
\end{proof}

\begin{remark}
The contact manifold $(Y_0, \xi_0)$ with monodromy $t_{b_1} t_{b_2} t_{b_3} t_{b_2}$ is obtained from that with monodromy $t_{b_1} t_{b_2} t_{b_3}$ by Legendrian surgery on $b_2$. As we noted, the monodromy $t_{b_1} t_{b_2} t_{b_3}$ yields $T^3$ and the embedding of this open book maps $b_2$ to a curve isotopic to $S^1 \times \{\text{pt}\}$ inside $T^3=S^1 \x T^2$. This curve is Legendrian and traverses the direction of the twisting of the contact planes, and has $tb = -1$. Surgery on $b_2$ yields a Seifert fibered $3$--manifold over $T^2$ with a single singular fiber of order $2$. 
\end{remark}

Next is our first main result, which will build on the example given in Lemma~\ref{keyexample}. We will then revamp our examples to an infinite family simply by combining them with Stein fillable planar open books. 

\begin{theorem} \label{mainthm1}
There are (infinitely many) contact $3$--manifolds with support genus one, each one of which admits infinitely many homotopy inequivalent Stein fillings, but do not admit arbitrarily large ones. These are all supported by genus--$1$ open books bounding genus--$1$ allowable Lefschetz fibrations on their Stein fillings with infinitely many distinct homology groups.
\end{theorem}

\begin{proof} 
We will construct a countable family of such examples, $(Y_m, \xi_m)$, 
\mbox{$m \in \N$}, each admitting a countably family of distinct Stein fillings $(X_{n,m}, J_{n,m})$,  $n\in \N$. All will be constructed from our principal example by taking a connected sum with some planar contact 3-manifold.

Let us begin with $(Y_0, \xi_0)= (Y, \xi)$. Capping all three boundaries of the genus--$1$ open book on $(Y, \xi)$ with monodromy $\phi$ we get a symplectic cobordism from $(Y, \xi)$ to a symplectic $T^2$--fibration \cite{Eliashberg3} with the induced monodromy $t_{b'_1} t_{b'_2} t_{b'_3} t_{b'_2}$, where $b'_i$ are the images of the curves $b_i$ on the closed genus--$1$ surface. By Remark~\ref{elliptic}, we already know that $t_{b'_1} t_{b'_2} t_{b'_3}$ can be completed to a positive factorization of the elliptic fibration on $E(1)$. So the factorization $t_{b'_1} t_{b'_2} t_{b'_3} t_{b'_2}$ can be completed to a positive factorization of an elliptic fibration on  the $\K$ surface. (Note that this means that we can embed all our Stein fillings $(X_n, J_n)$ into $\K$ surface.) That is, the contact $3$--manifold $(Y, \xi)$ admits a \emph{Calabi-Yau cap}, which by \cite{LiMakYasui} implies that the Betti numbers of Stein fillings of $(Y, \xi)$ is finite. In particular, it cannot admit arbitrarily large Stein fillings. 

Now let $(Y', \xi')$, where $Y' \neq S^3$, be any  contact $3$--manifold which admits a planar supporting open book. We simply take $(Y', \xi')$ to be $S^1 \x S^2$ with the canonical contact structure supported by an annulus open book with trivial monodromy $\phi'$, which is filled by the trivial allowable Lefschetz fibration with annulus fibers on $(X', J')$, the unique Stein filling $(S^1 \x D^3, J')$ of $(Y', \xi)$ \cite{Eliashberg2}.

By taking certain Murasugi sums of the open book $\phi$ with $m$ copies of $\phi'$, we can get a new genus--$1$ open book $\phi_m$ supporting a contact $3$--manifold $(Y_m, \xi_m)$, where we set $Y_m=Y \# \, m Y'$. By our assumption on $Y'$, we see that $\pi(Y_m)$ are all different;  for $Y'= S^1 \x S^2$, we have $\pi_1(Y_m )=(\Z^2) * (\Z^k)$. So $\{(Y_m, \xi_m) \, | \, m \in \N\}$ consists of infinitely many distinct elements. 

This Murasugi sum can be extended to the filling allowable Lefschetz fibrations, giving us a family of genus--$1$ allowable Lefschetz fibrations on the Stein fillings $(X_{n,m}, J_{n,m})$ of each $(Y_m, \xi_m)$, for $n \in \N$. Since 
\[H_1(X_{n,m})= H_1(X_n) \oplus H_1(X') = \Z^2 \oplus \Z_n \, ,\] 
these fillings are all distinct. Moreover, once again by Proposition~\ref{planarfillings}, all $(Y_m, \xi_m)$ have support genus $1$. However, $(Y_m, \xi_m)$ does not admit arbitrarily large Stein fillings: if it did, by \cite{Eliashberg2}  these would split as boundary connected sums of Stein fillings of $(Y_m, \xi_m)$ and copies of $(Y', \xi')$, respectively. By our pick of $(Y', \xi')$, this means $(Y_m, \xi_m)$ admits arbitrarily large Stein fillings, contradicting the observation we have made above. 
\end{proof}

\smallskip
\begin{remark} \label{YasuiExotic}
The first examples of contact $3$--manifolds with support genus one admitting infinitely many Stein fillings were given by Yasui in \cite{Yasui}. Yasui uses the same open book description for the standard contact structure on $T^3$ from the second author's thesis \cite{VHMthesis} to perform \emph{logarithmic transforms} and produce infinitely many pairwise exotic Stein fillings, as well as homologically distinct ones. These also admit genus--$1$ allowable Lefschetz fibrations, however their monodromies are different than ours and have many more Dehn twists. Unlike the small fillings from Lemma~\ref{keyexample} we tailored for the proof of the above theorem, most examples with larger topology (e.g. the fillings with distinct homologies in \cite{Yasui}) do not seem to embed into the $\K$ surface. It seems plausible that some of the applications in \cite{Yasui} can be modified to produce similar examples to ours, however the rather subtle partial conjugation employed in our construction of distinct positive factorizations arguably yields the most direct and explicit proof.
\end{remark}

\begin{remark}[Uniruled and Calabi-Yau caps]
As illustrated by our examples, admitting a Calabi-Yau cap \cite{LiMakYasui} does \emph{not} impose finiteness on possible homology groups of Stein or minimal symplectic fillings of a contact $3$--manifold. However, admitting a \emph{uniruled cap}  in the sense of \cite{LiMakYasui} might be sufficient for this, which would extend the case of contact $3$--manifolds with support genus zero we covered in Proposition~\ref{planarfillings}. 
\end{remark}

\smallskip
Our second main result will also build on a small variation of Lemma~\ref{keyexample}, this time to provide infinitely many fillings of quasipositive braids.

\begin{theorem} \label{mainthm2}
There are (infinitely many) elements in the $4$--strand braid group, each of which admits infinitely many quasipositive braid factorizations up to Hurwitz equivalence. A particular family gives infinitely many complex analytic annuli in the $4$--ball with different fundamental group complements, all filling the same \mbox{$2$-component} transverse link.
\end{theorem}

\begin{proof}
Recall the positive factorizations for the open book monodromy we had in Lemma~\ref{keyexample}.
Capping off the boundary component $\delta_3$ of $\Sigma_1^3$, we obtain another genus--$1$ open book  $\widehat{\phi}$ on a new contact $3$--manifold $(\widehat{Y}, \hat{\xi})$. It bounds allowable Lefschetz fibrations on $\widehat{X}_n$  with $4$ singular fibers, coming from positive factorizations 
\[\widehat{P}_n= (\, t_{\hat{b}_1} t_{\hat{b}_2} t_{\hat{b}_3} \, t_{t_{\hat{a}_2}^n(\hat{b}_2)} \,)^{t_{\hat{a}_2}^{-n}} \]
of $\widehat{\phi}$ in  $\Gamma_1^2$. Here $\hat{a}_i, \hat{b}_i, \delta_i$ denote the images of $a_i$, $b_i, \delta_i$ in the new page $\Sigma_1^2$; in particular, $\hat{a}_1=\hat{a}_3$. 

\begin{figure}[ht!] \label{curves}
  \begin{center}
   \includegraphics[width=13cm]{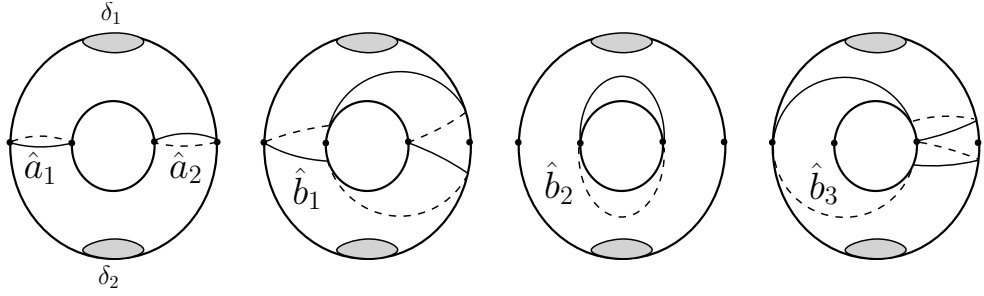}
    \end{center}

   \caption{Dehn twist curves $\hat{a}_i$, $\hat{b}_i$, $\delta_i$ on $\Sigma_1^2$, symmetric under the hyperelliptic involution with $4$ fixed points.}
   \end{figure}

The same homology arguments show that $\widehat{X}_n$ can be distinguished by their first homology,  $H_1(\widehat{X}_n; \Z)= \Z \, / n \, \Z$. For,  $H_1(\widehat{X}_n; \Z)$ is now generated by the homology classes of $\hat{a}_1, \hat{a}_2, \hat{b}_2$, with the relations 
\[  \hat{b}_2+\hat{a}_1+2\hat{a}_2=0 \, , \ \hat{b}_2=0 \, , \ \hat{b}_2 - \hat{a}_1- 2\hat{a}_2=0 \, ,   \ \hat{b}_2 -n\,\hat{a}_2=0  \]
induced by the vanishing cycles $\hat{b}_1, \hat{b}_2, \hat{b}_3$ and {$t_{\hat{a}_2}^{-n}(\hat{b}_2)$}. So $\hat{a}_2$ generates the whole group and is of order $n$.

One perk of the factorizations above is that they commute with a hyperelliptic involution $\iota$ on $\Sigma_1^2$, so each positive factorization $\widetilde{P}_n$ can be viewed as a degree four quasipositive factorization of the same element in the $4$-strand braid group $\mathcal{B}_4$. Topologically, the allowable  Lefschetz fibrations on $\widetilde{X}_n$ arise as double branched coverings along  \emph{positively braided surfaces} $S_n$ in the $4$--ball, all filling the same \mbox{$4$--strand} braid on the boundary. By the work of Rudolph, all $S_n$ are complex analytic curves \cite{Rudolph}. 

We easily see that for $n$ odd, the half-twisted bands connect all four horizontal disks in the band presentation, whereas for $n$ even there are two connected components. The Euler characteristic $\chi(S_n)= 4 - 4 = 0$ (the number of sheets minus the number of bands) so, $S_n$ is an annulus when $n$ is odd and is a link of punctured torus and a disk when $n$ is even. For either family of branched covers parametrized by odd $n$ or even $n$, since the double branched covers $\widetilde{X}_n$, which are all branched along homeomorphic surfaces $S_n$ in $D^4$ have different homologies, we necessarily have distinct $\pi_1(D^n \setminus S_n)$.

Lastly, we note that we could  get more elements in $\mathcal{B}_4$ with infinitely many factorizations with distinct $\pi_1$ complements by simply adding more  $t_{\hat{b}_2}$ factors to $\widehat{P}_n$. This would have no effect on the homology calculation in the cover since, the collection of vanishing cycles will be the same. It will however change the topology of $S_n$ and yield higher genera knotted complex surfaces filling in these links.
\end{proof}

\smallskip
\begin{remark}[Distinct positive factorizations in $\Gamma_1^k$] \label{AurouxFact}
The examples given in the proofs of Theorems~\ref{mainthm1} and~\ref{mainthm2} demonstrate that there are genus--$1$ mapping classes in $\Gamma_1^k$ for each $k \geq 2$, each one of which admits infinitely many inequivalent positive factorizations with four positive Dehn twists, all with distinct homologies. If we cap off one more boundary component, our induced factorizations in $\Gamma_1^1$ yield two distinct fundamental groups; $\pi_1=1$ or $\Z_3$. These complement the work of Auroux in \cite{Auroux}, where pairs of distinct factorizations into three or four positive Dehn twists, some with distinct homologies, were detected for a few other mapping classes in $\Gamma_1^1$.
\end{remark}

\begin{remark}[Quasipositive factorizations for low order braids] \label{GengFact}
There are examples of low order braids filled by different braided surfaces in the literature. The articles \cite{AurouxEtal} and \cite{Auroux} give examples of pairs of braided surfaces with distinct fundamental group complements, filling the same $3$-braids. These fillings have distinct topology ---one surface is connected and the other one isn't. Examples of pairs of \emph{connected} braided surfaces filling a $4$-braid were obtained by Geng in \cite{Geng}. Given Theorem~\ref{mainthm2}, one should ask whether an infinite family of braided surfaces with distinct fundamental group complements can possibly fill a $3$-braid. 
\end{remark}

Orevkov has recently shown that any $3$--braid admits at most  finitely many quasipositive factorizations up to Hurwitz equivalence \cite{O1}[Corollary~2]. Since $\Gamma_1^1$ is isomorphic to $\mathcal{B}_3$, where positive Dehn twists correspond to positive half-twists, any mapping class in $\Gamma_1^1$ also admits at most finitely many positive factorizations. (Also see \cite{CadavidEtal}.) We conclude that our examples of genus--$1$ mapping classes and braids with infinitely many positive and quasipositivefactorizations realize the smallest possible number of boundary components and number of strands, respectively.

\begin{remark} [Fillings of planar spinal open books]
\label{spinal}
An open question regarding contact $3$--manifolds supported by planar open books is if they admit at most finitely many Stein fillings. Here we can deduce a \emph{negative} answer for contact $3$--manifolds supported by planar \emph{spinal open books}, even though it is shown in \cite{LVW} that these behave very similar to planar open books in general. For example, their minimal symplectic fillings also admit planar Lefschetz fibrations over more general compact surfaces (the topology of which are determined by the spine of the open book), and they all turn out to be Stein fillings. (The reader can turn to the Appendix of \cite{BV1} for an overview of spinal open books introduced in \cite{LVW}.) In turn, our homology arguments in \cite{BMV}[Proposition~1] work the same for the induced positive factorizations, since the commutators in these factorizations die in the homology. We therefore conclude that contact $3$--manifolds supported by planar spinal open books, just like the planar case, will not admit arbitrarily large Stein fillings.

Here is a sketch of the construction of these examples: consider the simplest monodromy we had for the annuli fillings in Theorem~\ref{mainthm2}. If we take the 3rd power of it, we get a pure $4$--braid, whereas the 3rd power of the positive factorizations will still be distinct. The latter is because the collection of Dehn twists is the same in double branched covers, so the fillings can still be distinguished by the same homology calculation. If we now drill out a fibered tubular neighborhood of the new braided surfaces, the boundary is a planar spinal open book, whereas the interior is the complement of a complex analytic genus $3$ curve (filling the pure braid), which will yield a minimal symplectic filling. By \cite{LVW}, these are all Stein fillings. 
\end{remark}

\vspace{0.1in}

\end{document}